\documentclass[11pt]{amsproc}
 \usepackage[margin=1in]{geometry}
\usepackage{setspace,fullpage}
\usepackage{graphicx}
\usepackage[nice]{nicefrac}
\usepackage{amssymb}
\usepackage{multirow}
\usepackage{array}

\usepackage{amsmath,amsthm,amscd,amssymb, mathrsfs}
\usepackage{latexsym}

\numberwithin{equation}{section}

\theoremstyle{plain}
\newtheorem{theorem}{Theorem}[section]
\newtheorem{lemma}[theorem]{Lemma}

\newtheorem{proposition}[theorem]{Proposition}
\newtheorem{conjecture}[theorem]{Conjecture}

\theoremstyle{definition}
\newtheorem{definition}[theorem]{Definition}

\theoremstyle{remark}
\newtheorem{remark}[theorem]{Remark}

\newtheorem{case[theorem]}{Case}

\newtheorem*{notation}{\textbf{Notation}}

\title[\parbox{14cm}{\centering{Restriction estimates for the flat disks over finite fields \hspace{1in}}} \quad]{Restriction estimates for the flat disks over finite fields }
\author{Doowon Koh}

\address{Department of Mathematics\\
Chungbuk National University \\
Cheongju, Chungbuk 28644 Korea}
\email{koh131@chungbuk.ac.kr}

\thanks{Key words and phrases: Finite field, Restriction problem,  Flat disk\\
Doowon Koh would like to thank the Department of Mathematics at the University of Rochester for hospitality during a visit where this work was completed. This work was supported by a funding for the academic research program of Chungbuk National University in 2022, and   Basic Science Research Programs through National Research Foundation of Korea (NRF) funded by the Ministry of Education (NRF-2018R1D1A1B07044469).} 

\subjclass[2010]{42B05, 43A32, 43A15 }

\begin{document} 

\begin{abstract}  In this paper we study the restriction estimate for the flat disk over finite fields.
 Mockenhaupt and Tao initially studied this problem but  their results were addressed only for dimensions $n=4, 6.$  We improve and extend their results  to all dimensions $n\ge 6.$  More precisely, we obtain the sharp $L^2\to L^r$ estimates, which cannot be proven by applying the usual Stein-Tomas argument  over a finite field even with the optimal Fourier decay estimate on the flat disk. One of main ingredients is to discover and analyze an explicit form of the  Fourier transform of the surface measure on the flat disk. 
In addition, based on the recent results on the restriction estimates for the paraboloids,  we  address improved restriction  estimates for the flat disk beyond the $L^2$ restriction  estimates.

\end{abstract}

\maketitle
\section{Introduction} 
Let $\mathbb F_q^n$ be an $n$-dimensional vector space over a finite field $\mathbb F_q$ with $q$ elements, where $q$ is odd.
Let $n=2d$ for an integer $d\ge 2.$
In this paper we investigate the restriction estimate for the following variety
\begin{equation}\label{DefFlat}\mathcal{F}:=\{(\alpha,  \alpha\cdot \alpha,  \beta,  \alpha\cdot \beta) :  \alpha, \beta\in \mathbb F_q^{d-1}\} \subset \mathbb F_q^{2d},\end{equation}
where $\alpha \cdot \beta$ is the usual inner product of $\alpha$ and $\beta$.  This variety $\mathcal{F}$ is referred to as the flat disk over a finite field.\\

In 2004, Mockenhaupt and Tao \cite{MT04} initially studied the finite field analogue of the restriction problem for various algebraic varieties including the flat disk over finite fields.  Since their work was introduced,  follow-up studies have been extensively conducted, but most of them have focused on hyper-surfaces such as the paraboloid,  the sphere, and the cone  (see, for instance,  \cite{IK09,  LL10, LL13, IKL20, SZ17, RS18, IK10, KPV18, KLP22}).  On the other hand,  there are only few known concrete results on the restriction problem for the flat disk over the finite field.  For example,  Mockenhaupt and Tao \cite{MT04}  addressed some partial results in the specific case when $n=6,$ and settled it when $n=4.$ \\

The purpose of this paper is to present a systematic study on restriction theory related to the flat disk, and   improved  results in all dimensions $n\ge 6.$  To this end, we begin by recalling notation regarding the restriction problem for the flat disk.
We endow $\mathbb F_q^n$ with counting measure  ``$dm$''.
Let $d\sigma$ be normalized ``surface measure'' on the flat disk $\mathcal{F}$ living in the dual space of $\mathbb F_q^n:$
$$ \int_{\mathcal{F}} f(x) d\sigma(x):= \frac{1}{|\mathcal{F}|} \sum_{x\in \mathcal{F}} f(x).$$
Here, we note that $|\mathcal{F}|=q^{n-2}$, which denotes the cardinality of $\mathcal{F}.$ The normalized surface measure $d\sigma$ can  be interpreted as the following:
$$ d\sigma(x)= \frac{q^n}{|\mathcal{F}|} 1_\mathcal{F}(x) dx=q^2 1_\mathcal{F}(x) dx,$$
where $1_{\mathcal{F}}$ denotes the indicator function of $\mathcal{F}$ and   we endow the dual space of $\mathbb F_q^d$ with
 normalized counting measure ``$dx$''.  Hence, we can identify $d\sigma$ as a function $q^2 1_\mathcal{F}$ on the dual space of $\mathbb F_q^n.$
  \\

For $1\le p, r\le \infty,$ we define $R^*_{\mathcal{F}}(p\to r)$ to be the smallest constant such that 
the extension estimate 
$$ || (fd\sigma)^\vee||_{L^r(\mathbb F_q^n, dm)}\le R^*_{\mathcal{F}}(p\to r) ||f||_{L^p(\mathcal{F}, d\sigma)}$$
holds true for all complex valued functions $f$ on $\mathcal{F}.$ Here, the inverse Fourier transform of the measure $fd\sigma$ is defined by
$$  (fd\sigma)^\vee(m)=\int \chi(x\cdot m) f(x) d\sigma(x) =\frac{1}{|\mathcal{F}|} \sum_{x\in \mathcal{F}} \chi(x\cdot m) f(x),$$
where $\chi$  denotes the canonical additive character of $\mathbb F_q$ (see Definition \ref{def1}).
 By duality,  $R^*_{\mathcal{F}}(p\to r)$ is the same as the smallest constant such that
the restriction estimate
\begin{equation}\label{Duality} ||\widehat{g}||_{L^{p'}(\mathcal{F}, d\sigma)} \le R^*_{\mathcal{F}}(p\to r) ||g||_{L^{r'}(\mathbb F_q^n, dm)}\end{equation}
holds for all functions $g$ on $\mathcal{F},$ where $p'$ denotes the H\"older conjugate of $p,$ namely  $1/p+ 1/p'=1.$
The proof of the duality over a finite field can be found in Theorem 4.1 of Appendix in \cite{IKLPS}.  
Recall that the Fourier transform of $g$, denoted by $\widehat{g},$  is defined by
$$ \widehat{g}(x)=\sum_{m\in \mathbb F_q^d}  \chi(-x\cdot m) g(m).$$
The restriction problem for $\mathcal{F}$ is to determine all exponents $1\le p, r \le \infty$ such that 
$$ R^*_{\mathcal{F}}(p\to r) \lesssim 1.$$
Here, and throughout this paper, we use $A\lesssim B$ if there is a constant $C$ independent of $q$ such that $ A\le C B.$
We also use the notation $A\sim B$ if $A\lesssim  B$ and $B\lesssim A.$ \\

Similar to the definition of $R^*_{\mathcal{F}}(p\to r),$  one can define  $R_V^*(p\to r)$  for any algebraic variety $V$ in $\mathbb F_q^\alpha, \alpha \ge 2.$  
We say that the $L^p\to L^r$ estimate for $V$ holds if $R_V^*(p\to r)\lesssim 1.$

\begin{remark}\label{rem1}Using  H\"older's inequality and the nesting properties of $L^p$-norms over finite fields, one can observe that
$R_V^*(p_1\to r) \le R_V^*(p_2\to r) $ for $\ p_1\ge p_2, $ and $R_V^*(p\to r_1)\le R_V^*(p\to r_2)$ for $ r_1\ge r_2,$
which will allow us to reduce the proofs of our results to certain endpoint estimates.
\end{remark}
 
Over the last few decades, various methods have been developed in the study of  the restriction problem in the Euclidean setting, but  only a few of them have been applied to that in the finite field setting. 
Among such applicable methods, the most well-known method is the Stein-Tomas argument, which enables us to deduce  the ``$r$'' index such that $R^*_V(2\to r)\lesssim 1.$ To be precise, Mockenhauput and Tao \cite{MT04} addressed the following formula (see  the paragraph given after the proof of Lemma 1.6 in \cite{MT04}).
\begin{lemma}  [\cite{MT04}] \label{STA}
Let $d\sigma_v$ denote normalized surface measure on the algebraic variety $V$ in $\mathbb F_q^\alpha, \alpha\ge 2.$ 
Suppose that for some integers $0<k, s<\alpha,$  it satisfies that 
$$ |V|\sim q^s \quad \mbox{and}\quad  \max_{m\in \mathbb F_q^\alpha\setminus \{\mathbf{0}\}} |(d\sigma_v)^\vee (m)|\lesssim q^{-\frac{k}{2}}.$$
Then  $R^*_V(2\to r)\lesssim 1$ whenever   $r\ge 2+ \frac{4(\alpha-s)}{k}.$
\end{lemma}

One curious question that naturally arises from the above lemma is whether the value of $k,$ which gives the optimal Fourier decay estimate on the surface measure, yields the optimal $r$ index for the $L^2\to L^r$ estimate for $V$. This question is the same as follows.\\
If  $|V|\sim q^s$ and   $\max\limits_{m\in \mathbb F_q^\alpha\setminus \{\mathbf{0}\}} |(d\sigma_v)^\vee (m)|\sim q^{-\frac{k}{2}},$ then  for all $ 1\le r < 2+ \frac{4(\alpha-s)}{k},$   is it impossible that  $R^*_V(2\to r)\lesssim 1$?
In the finite field setting, it turns out that the answer is, in general, ``No''.  
For instance,  let $d\sigma_P$ be the normalized surface measure on the paraboloid  $P$ in $\mathbb F_q^d, d\ge2,$ 
\begin{equation}\label{DefP} P:=\{x\in \mathbb F_q^d: x_1^2+ x_2^2+ \cdots + x_{d-1}^2= x_d\}.\end{equation}
Mockenhaupt and Tao \cite{MT04} observed that  for all dimensions $d\ge 2,$  
$$ |P|=q^{d-1} \quad \mbox{and}\quad  \max_{m\in \mathbb F_q^d\setminus \{\mathbf{0}\}} |(d\sigma_P)^\vee (m)|\sim q^{-\frac{d-1}{2}}.$$
Hence, if we invoke Lemma \ref{STA}, by taking $\alpha=d,  s=k=d-1,$  then we obtain
\begin{equation}\label{STT} R^*_P\left(2\to \frac{2d+2}{d-1}\right) \lesssim 1.\end{equation}
This result is referred to as the Stein-Tomas result, which provides the sharp $L^2\to L^r$ exponent for the piece of the paraboloid in the Euclidean case in the sense that  the exponent $r$ for the estimate $R^*(2\to r)\lesssim 1$ cannot be lower than $\frac{2d+2}{d-1}.$
However, in the finite field setting,  the Stein-Tomas exponent is not sharp except for  the following specific cases:
\begin{itemize}
\item $d\ge 3$ is odd and $-1$ is a square number in $\mathbb F_q.$
\item $d=4\ell +1$ for $\ell\in \mathbb N$, and $-1$ is not a square number in $\mathbb F_q,$
\end{itemize}
For any other cases including the even dimensions $d\ge 2,$ the Stein-Tomas result can be significantly improved to much smaller exponents (for example, see Theorems \ref{restEven} and \ref{restOdd} in Section \ref{Sec5}).\\

In the study of the restriction problem for algebraic varieties over finite fields,  there are currently two important research trends.\\

Regarding the above question and examples, one of the main concerns  is to verify  the optimal $L^2\to L^r$ estimate for a variety $V,$ where   the Stein-Tomas argument, Lemma 1, fails to yield the optimal $L^2\to L^r$ estimate for $V.$ 
Our first result below is closely related to this issue   for the flat variety $\mathcal{F}$ in $\mathbb F_q^n.$ We establish the sharp $L^2\to L^r$  restriction estimate for  the flat disk.
\begin{theorem}  \label{mainI}
Let $\mathcal{F}$ be the flat disk in $\mathbb F_q^n,  n=2d\ge 4,$ defined as in \eqref{DefFlat}. Then we have
$$ R^*_{\mathcal{F}} \left(2\to \frac{2n+4}{n-2}\right) =R^*_{\mathcal{F}} \left(2\to \frac{2d+2}{d-1}\right)\lesssim 1.$$
\end{theorem}
\begin{itemize}
\item By the nesting property of norms over finite fields (see Remark \ref{rem1}),  Theorem \ref{mainI} implies that  $R^*_{\mathcal{F}} (2\to r) \lesssim 1$ for all $r$ with  $\frac{2n+4}{n-2}=\frac{2d+2}{d-1}\le r\le \infty.$ Moreover, it provides the sharp $L^2\to L^r$ restriction estimate for the flat disk. This can be shown from Lemma \ref{NecC} in the following section. More precisely, the proof of the sharpness will be given in Remark \ref{sharpexa}.

\item Theorem \ref{mainI} cannot be obtained from a direct application of the Stein-Tomas argument, namely Lemma \ref{STA}. Indeed,  it follows from Proposition \ref{Promain} that
$ \max_{m\in \mathbb F_q^n\setminus \{\mathbf{0}\}} |(d\sigma)^\vee (m)|= q^{-\frac{k}{2}}$ with $k=\frac{n-2}{2}.$
In addition, we see that $|\mathcal{F}|=q^{n-2}.$ Hence,  applying Lemma \ref{STA} with $s=n-2, k=\frac{n-2}{2},$  we get
$R^*_{\mathcal{F}}(2\to \frac{2n+12}{n-2})\lesssim 1,$ which, however, is much weaker than Theorem \ref{mainI}.

\item As a main idea to derive Theorem \ref{mainI},  we decompose the surface measure on $\mathcal{F}$ as the 6 surface measures such that they have pairwise disjoint supports and each of them has a constant Fourier decay on the support, which makes our analysis much more efficient and simpler.
\end{itemize}

\begin{remark}
One of main ingredients to deduce Theorem \ref{mainI} is based on  the explicit Fourier transform on the surface measure $d\sigma$ of the flat disk $\mathcal{F},$ which will be given as Proposition \ref{Promain}. The key idea to compute it is to view the flat disk $\mathcal{F}$ in \eqref{DefFlat}  as the set of common solutions of the following two equations:
\begin{equation} \label{defFeq} x_d=x_1^2+x_2^2+\cdots + x_{d-1}^2, \quad  x_{2d}= x_1x_{d+1} +x_2x_{d+2} + \cdots+ x_{d-2}x_{2d-2}+ x_{d-1} x_{2d-1}.\end{equation} 
Then, adapting the argument by the discrete Fourier analysis due to Iosevich and Rudnev \cite{IR07}, we are able to relate the equations defining the flat disk to certain exponential sums, which essentially reduce to the well-understood Gauss sums.

\end{remark}

 The other interesting research trend is to deduce  a new $L^p\to L^r$ restriction estimate for $V$ such that  the exponent $p$ is not based on ``$2$''.  Here, and throughout, a new $L^p\to L^r$ restriction result  means any  restriction result which cannot be obtained as a direct corollary of the optimal $L^2\to L^r$ estimate for $V.$ \\ 
 
In the Euclidean setting,  various methods have been developed to induce new $L^p\to L^r$ estimates (see, for example, \cite{To75, Wo01, To03, Gu16, HR19, Wa21}).  However,  there are many limitations to the application of such techniques in the finite field.
It has been considered as an extremely hard problem to deduce  a new $L^p\to L^r$ result for an algebraic variety. However,   Mockenhaupt and Tao \cite{MT04} proposed a new method to deduce a new $L^p\to L^r$ estimate for the flat disk $\mathcal{F}$ in $\mathbb F_q^n$ with $n\ge 4$ even.  More precisely, they related the problem to the Kakeya maximal estimate and the restriction estimates for the paraboloids in $\mathbb F_q^{n/2}.$ As a result, they addressed a new $L^{36/13}\to L^{18/5+ \varepsilon}$ estimate for any $\varepsilon >0$ in the case when $n=6$ and $-1$ is not a square number of $\mathbb F_q.$\\

Since the Mockenhauput and Tao's work \cite{MT04},  much improvement on the restriction estimate for paraboloids has been made. Furthermore,  the maximal Kakeya conjecture over finite fields was settled by Ellenberg-Oberlin-Tao \cite{EOT10}.  Hence,  improved  new $L^p\to L^r$ restriction estimates should be obtained. As such  possible improvements have not been addressed in the literature,  in this paper we will  clearly indicate the improved new $L^p\to L^r$ restriction results for the flat disk.  To this end, we will formulate a proposition, which makes it possible to derive new $L^p\to L^r$ restriction estimates for the flat disk in $\mathbb F_q^{2d}$ directly from restriction results for paraboloids in $\mathbb F_q^d$ (see Proposition \ref{LemFormula}). 
\\

Now we state  our new $L^p\to L^r$ restriction estimates for the flat disk $\mathcal{F}$ in $\mathbb F_q^{n}=\mathbb F_q^{2d}, d\ge 2.$  For even integers $d\ge 2,$ we have the following consequences.
\begin{theorem}\label{main2} Let $\mathcal{F}$ be the flat disk in $\mathbb F_q^{2d}.$ 
\begin{enumerate}
\item [(1)] If $d=2,$ then $R^*_{\mathcal{F}}(4\to 4)\lesssim 1.$
\item [(2)] If $d=4,$ then $R^*_{\mathcal{F}}\left(\frac{28}{11}\to \frac{28}{9}\right)\lesssim 1.$
\item [(3)] If $d=4$ and $q$ is prime, then $R^*_{\mathcal{F}} \left(\frac{18}{7} \to 3\right)\lesssim 1$
\item [(4)] If $d=6,$ then $R^*_{\mathcal{F}}\left(\frac{80+30\varepsilon}{34+15\epsilon}\to \frac{8}{3}+ \varepsilon \right)\lesssim 1$ for all $\varepsilon >0.$
\item [(5)] If $d\ge 8$ is even, then $R^*_{\mathcal{F}}\left(\frac{2d^2+2d-4}{d^2-2} \to \frac{2d+4}{d}\right)\lesssim 1.$
\end{enumerate}
\end{theorem}

For odd integers $d\ge 3,$  we obtain the following restriction estimates. 
\begin{theorem}\label{main3}
Let $\mathcal{F}$ be the flat disk in $\mathbb F_q^{2d}.$ 
\begin{enumerate}
\item [(1)] If $d=3$ and $q\equiv 3 \pmod{4},$ then  $R^*_{\mathcal{F}} \left( \frac{36-10\varepsilon}{13-5\varepsilon}\to \frac{18}{5}-\varepsilon \right) \lesssim 1$ for some $\varepsilon >0.$
\item [(2)] If $d=3$ and $q\equiv 3 \pmod{4}$ is prime,  then  $R^*_{\mathcal{F}}\left(\frac{376+106\varepsilon}{135+53\varepsilon} \to \frac{188}{53}+\varepsilon \right)\lesssim 1$ for all $\varepsilon >0.$
\item [(3)] If $d\ge 3 $ is odd and $q\equiv 1 \pmod{4},$ then  $R^*_{\mathcal{F}}\left(\frac{2d+2}{d} \to \frac{2d+2}{d-1}\right) \lesssim 1.$
\item [(4)]  If $d=  4\ell+1$ with $\ell \in \mathbb N,$ and $q\equiv 3 \pmod{4},$ then  $R^*_{\mathcal{F}}\left(\frac{2d+2}{d}  \to \frac{2d+2}{d-1}\right) \lesssim 1.$
\item [(5)]   If $d=  4\ell+3$, with $\ell \in \mathbb N,$ and $q\equiv 3 \pmod{4}, $ then  $R^*_{\mathcal{F}}\left(\frac{2d^2+2d-4}{d^2-2} \to \frac{2d+4}{d}\right) \lesssim 1.$
\end{enumerate}
\end{theorem}
\begin{itemize} 
\item 
Notice from Remark \ref{rem1} that  the smaller exponent implies the better restriction result for $\mathcal{F}$ in $\mathbb F_q^n, $ with $n=2d\ge 4.$  In order to deduce further results from a known restriction estimate, one can invoke the interpolation theorem (Theorem \ref{tmR}) with the trivial estimate $R^*_{\mathcal{F}} (1\to \infty)\lesssim 1.$ 
Hence,  Conjecture \ref{Conj} in the following section shows that  to settle the restriction problem for the flat disk $\mathcal{F}$ in $\mathbb F_q^{n},$ with $n=2d\ge 4,$  it suffices to establish the critical endpoint estimate:
$$R^*_{\mathcal{F}}\left(\frac{2n}{n-2} \to \frac{2n}{n-2}\right) = R^*_{\mathcal{F}}\left(\frac{2d}{d-1} \to \frac{2d}{d-1}\right) \lesssim 1.$$

\item Observe that the first part of Theorem \ref{main2} is the solution of the restriction problem for the flat disk $\mathcal{F} \subset \mathbb F_q^4.$ This was first pointed out by Mockenhaupt and Tao \cite{MT04} but the detail proof was not given. 

\item The third and  fourth parts of Theorem \ref{main3} are not new $L^p\to L^r$ restriction estimates for $\mathcal{F}$ as the sharp $L^2\to L^r$ result, Theorem \ref{mainI}, already implies those results. However, all other results including Theorem \ref{main2} are new $L^p\to L^r$ estimates.
\item
As we will see from Conjecture \ref{Conj},  the conjectured exponents for $R^*_{\mathcal{F}}(p\to r)$ to be bounded  are irrelevant of the ground field $\mathbb F_q.$  Hence,  it is natural to expect that  one can remove  the conditions on $q$ appearing in the statement of Theorem \ref{main3}.

\item One can obtain further results by interpolating  the  sharp $L^2\to L^r$ estimate of Theorem \ref{mainI}  and   the results of Theorems  \ref{main2} and \ref{main3}.  For example,  the  previously known estimate $R^*_{\mathcal{F}}(\frac{36}{13}\to \frac{18}{5}+\varepsilon) \lesssim 1$ for  $n=6$ (or $d=3$), due to Mockenhaupt and Tao,   can be improved to $R^*_{\mathcal{F}}(\frac{36}{13}\to \frac{72}{20+5\varepsilon})\lesssim 1,$ which follows by interpolating  the first part of Theorem \ref{main3} and  the result  $R^*_{\mathcal{F}}(2\to 4)\lesssim 1,$ which is Theorem \ref{mainI} for $d=3.$
\end{itemize}

\begin{remark} 
Theorems \ref{main2} and \ref{main3} are much weaker than the conjectured results (Conjecture \ref{Conj}) except for the first result of Theorem \ref{main2}. We notice that one cannot settle this question by using our method in this paper. As we shall see, our results follow by applying the $L^2\to L^r$ restriction estimate for the paraboloid in $\mathbb F_q^{n/2}$ (see Proposition \ref{LemFormula}). Even using the optimal $L^2\to L^r$ restriction estimate for the paraboloid, it fails to produce the conjectured result (see, for example, Remark \ref{rem53}). For this reason, it leaves the question of finding a new methodology to completely solve this problem.
In addition, it would be interesting to extend our work to general quadratic surfaces of co-dimension bigger than one. In the Euclidean case, such problems have been extensively studied. We refer the reader to \cite{Ch85, BL04, BLL17, LL19}. 
However, it seems that the Euclidean arguments do not work in the finite field case.
\end{remark}

The remaining part of this paper will be essentially designed to give proofs of our main results.
\begin{notation} 
Throughout this paper, we will use the following notation:
\begin{itemize}
\item  For any integer $\alpha\ge 1$,   we use  $\mathbf{0}\in \mathbb F_q^\alpha$ to denote the zero vector in $\mathbb F_q^\alpha.$  In particular,  we write $0$ for $\mathbf{0}$ when $\alpha=1.$  We write $\delta_{\mathbf{0}}$ for the indicator function of $\{\mathbf{0}\},$ namely,   $\delta_{\mathbf{0}}(\alpha)=1$ for $\alpha=\mathbf{0},$ and $0$ otherwise.
\item For a vector $m$ in $\mathbb F_q^\alpha,$  we write $m_j$ to denote the $j$-th coordinator of $m.$  For example, we have
$m=(m_1, \ldots, m_\alpha)\in \mathbb F_q^\alpha.$    We also define 
$$ ||m||:= \sum_{j=1}^\alpha m_j^2.$$
\item For a simple notation, we identify a set $E$ with the indicator function $1_E$ of the set $E,$  where $1_E(x)=1$ for $x\in E,$ and $0$ otherwise.  For example, we write $\widehat{E}$ for $\widehat{1_E}$, the Fourier transform  of the indicator function $1_E.$

\item  For $1\le r\le \infty,$   the H\"older conjugate of $r$  is denoted by
$r'$, namely,  $1/r + 1/r'=1.$
\end{itemize}
\end{notation}

\section{Preliminaries}

Some necessary conditions for $R^*_V(p\to r)$ to be bounded can be determined by the size of the underlying variety $V$ and  the size of any maximal affine subspace lying in $V.$ 
\begin{lemma}[Mockenhaupt-Tao, \cite{MT04}]\label{MTNe} Let $V$ be an algebraic variety in $\mathbb F_q^\alpha, \alpha \ge 2,$  with  the size $|V|=q^s,  0< s< \alpha.$
In addition, assume that  the variety $V$ contains an affine subspace $H$ with the size $|H|=q^k,  0 <k\le s.$ 
If $R^*_V(p\to r)\lesssim 1$ for some $1\le p, r\le \infty,$  then we have
$$ r\ge \frac{2\alpha}{s} \quad\mbox{and}\quad r\ge \frac{p(\alpha-k)}{(p-1)(s-k)}.  $$
\end{lemma}
\begin{proof}
The proof of the above lemma can be found on  pages 41--42 in \cite{MT04}. 
\end{proof}

\subsection{Necessary conditions for the bound of $R^*_{\mathcal F}(p\to r)$}
From now on we always assume that the flat disk $\mathcal{F}$ is the variety lying in $\mathbb F_q^n$ with $n=2d\ge 4$ even integer.
\begin{lemma}\label{NecC}  Suppose that $R^*_{\mathcal{F}}(p\to r)\lesssim 1$ for some $1\le p,r \le \infty.$ Then we have
$$ r \ge \frac{2n}{n-2} \quad \mbox{and}\quad r \ge \frac{p(n+2)}{(p-1)(n-2)}.$$
\end{lemma}
\begin{proof} Since $\mathcal{F}$ is contained in $\mathbb F_q^n$, it is not hard to see that  the cardinality of $\mathcal{F}$ is $q^{n-2}$.  We note that when $n=2d, d\ge 2,$ the flat disk $\mathcal{F}$ is the set of common solutions of the following equations:
$$ x_d=x_1^2+x_2^2+\cdots + x_{d-1}^2, \quad  x_{2d}= x_1x_{d+1} +x_2x_{d+2} + \cdots+ x_{d-2}x_{2d-2}+ x_{d-1} x_{2d-1}.$$ 

Setting $H=\{{\bf 0}\}\times \mathbb F_q^{d-1}\times \{0\} \subset  \mathbb F_q^d \times \mathbb F_q^{d-1} \times \mathbb F_q,$ it is easily seen that $H$ satisfies that $|H|=q^{d-1}=q^{\frac{n-2}{2}},$ and  is a subspace lying on the flat disk $\mathcal{F}.$ Hence, invoking Lemma \ref{MTNe} with taking   $\alpha=n, s=n-2,$ and $k=(n-2)/2$,   we obtain the required necessary conditions for  the estimate $R^*_{\mathcal{F}}(p\to r)\lesssim 1.$
\end{proof}
It can be conjectured that the above necessary conditions are in fact sufficient conditions for the bound of $R^*_{\mathcal{F}}(p\to r).$ In other words,  we conjecture the following statement.

\begin{conjecture} \label{Conj} If $(1/p, 1/r)$ is contained in the convex hull of points $(0, 0)  (0, \frac{n-2}{2n}),  (\frac{n-2}{2n},  \frac{n-2}{2n}), $ and $(1,0),$ then  
$$ R^*_{\mathcal{F}}(p\to r) \lesssim 1.$$
\end{conjecture}

\begin{remark}\label{sharpexa}Theorem \ref{mainI} shows that  the above conjecture  is true  when the $p$ index is 2. Indeed,  letting $x=1/p,  y=1/r$,  the line segment passing through  $(1, 0)$ and $ (\frac{n-2}{2n}, \frac{n-2}{2n})$ is given by the equation
$$ y= \frac{-(n-2)x}{n+2} + \frac{n-2}{n+2}, \quad \left(\frac{n-2}{2n}\le x\le 1\right).$$
Hence, when $x=1/p=1/2,$   it follows by a direct computation that   $y=1/r=\frac{n-2}{2n+4},$ which is exactly corresponding to Theorem \ref{mainI}. \end{remark}

\subsection{The standard Gauss sum}

One of the ingredients to prove Theorem \ref{mainI}  will be to  deduce an explicit Fourier transform of the surface measure $d\sigma$ on the flat disk $\mathcal{F}$ in $\mathbb F_q^n.$  To do this,  by means of the Fourier analysis over finite fields, we shall reduce the matter to the estimate of an explicit Gauss sum. Here we review  basics for the Gauss sum.\\

We begin by reviewing the definition of the canonical additive character of $\mathbb F_q$, which is given in \cite{LN97}.
Let $p$ be the characteristic of $\mathbb F_q$ with $q=p^\ell$ for some positive integer $\ell.$ The absolute trace function  $Tr$ is a function from $\mathbb F_q$ to $\mathbb F_p,$ defined by
$$ Tr(t)=t+ t^p+ t^{p^2} + \cdots + t^{p^{\ell-1}}.$$
 It is shown in  Section 3  of \cite{LN97} that  the absolute trace function is well-defined.
\begin{definition} [Canonical additive character and the quadratic character, \cite{LN97}] \label{def1}
 The function $\chi$ defined by
 $$\chi(c)=e^{2\pi i Tr(c)/p} \quad \mbox{for}~~ c\in \mathbb F_q$$
 is called the canonical additive character of $\mathbb F_q.$  On the other hand,  the multiplicative character $\eta$ is a function from $\mathbb F_q^* \to \mathbb{R},$ defined by
 $$ \eta(a)=\left\{ \begin{array}{ll} 1 \quad &\mbox{if}~a ~\mbox{is a square number of}~ \mathbb F_q^*,\\
                                                      -1 \quad &\mbox{if}~ a ~\mbox{is not a square number of }~ \mathbb F_q^*.\end{array}\right.$$
\end{definition}

It is known that  $ \eta(-1)=-1$ if and only if  $q\equiv 3 \pmod{4},$  and   $\eta(-1)=1$ if and only if $q\equiv 1 \pmod{4}$ (for example, see Remark 5.13, \cite{LN97}).\\

The orthogonality of characters $\chi$ and $\eta$ states that
$$\sum_{t\in \mathbb F_q} \chi(at)=\left\{\begin{array}{ll}  q \quad\mbox{if}~~ a=0,\\
                                                                                           0 \quad\mbox{if}~~ a\ne 0,\end{array}\right.
                \quad \mbox{and}\quad \sum_{t\in \mathbb F_q^*} \eta(at)=0  \quad\mbox{if}~~ a\ne 0.$$

  The standard Gauss sum determined by $\chi$ and $\eta$  is defined by
  $$\mathcal{G}=\mathcal{G}(\eta, \chi):= \sum_{t\in \mathbb F_q^*} \eta(t) \chi(t).$$
 We will invoke the following well-known property of the standard Gauss sum.  Here we provide an elementary proof.
 \begin{lemma}\label{SGa} We have
 $$ \mathcal{G}^2=\mathcal{G}(\eta, \chi)^2=\eta(-1) q.$$
  \end{lemma}
\begin{proof}
Since $\eta =\overline{\eta}$ and $\overline{\chi(t)}= \chi(-t),$   it is seen by a change of variables that
$$ \mathcal{G}(\eta, \chi)=\eta(-1) \overline{\mathcal{G}(\eta, \chi)}.$$ 
Hence,  $\mathcal{G}(\eta, \chi)^2= \eta(-1) |\mathcal{G}(\eta, \chi)|^2,$ so the problem is reduced to showing that $|\mathcal{G}(\eta, \chi)|^2=q.$
Indeed,  it follows that
$$|\mathcal{G}(\eta, \chi)|^2= \left(\sum_{a\ne 0} \eta(a)\chi(a)\right) \left(\sum_{t\ne 0} \overline{\eta(t) \chi(t)} \right) =\sum_{a, t\ne 0} \eta(at^{-1}) \chi(a-t).$$
By a change of variables,  letting $b=at^{-1}$ for any fixed $t\ne 0,$ 
$$=\sum_{t\ne 0} \sum_{b\ne 0} \eta(b) \chi( (b-1)t)=\sum_{b\ne 0} \eta(b) \left( -1+ \sum_{t\in \mathbb F_q} \chi((b-1))t\right).$$
By the orthogonality of $\chi$ and $\eta$,  we obtain the required estimate
$ |\mathcal{G}(\eta, \chi)|^2=q,$ where we also used the simple fact that $\eta(1)=1.$ This completes the proof of the lemma.
\end{proof}

It is not hard to note that  for any non-zero  $a\in \mathbb F_q^*,$ 
$$ \sum_{t\in \mathbb F_q} \chi(at^2)=\eta(a) \mathcal{G}.$$
Completing the square and using a simple change of variables,  the above formula can be generalized to the formula below:
For any non-zero $a\in \mathbb F_q^*$ and  any $b\in \mathbb F_q,$   we have
\begin{equation}\label{CSQ} \sum_{t\in \mathbb F_q} \chi(at^2+bt) = \eta(a) \mathcal{G} \chi\left(\frac{b^2}{-4a}\right).\end{equation}

\subsection{Ellenberg-Oberlin-Tao  Kakeya maximal theorem over finite fields} 
We review the connection between the restriction estimate for the flat disk $\mathcal{F}$ and the Kakeya maximal estimate over finite fields. 
To deduce Theorem \ref{main2} and Theorem \ref{main3}, we will heavily use  the connection as well as recently established restriction estimates for paraboloids in $\mathbb F_q^d.$ \\

We begin with some notation  related to the Kakeya maximal problem over finite fields.
Consider a direction $v\in \mathbb F_q^{d-1}, d\ge 2,$ and a vector $z_0\in \mathbb F_q^{d-1}.$ Then  the line $l(z_0, v)$ in $\mathbb F_q^d$ is defined by 
$$ l(z_0, v):= \{ (z_0+ t v,  t): t\in \mathbb F_q\} \subset \mathbb F_q^d.$$
For a function $h$ on $\mathbb F_q^d$,  the Kakeya maximal function $h^*$ is defined on $\mathbb F_q^{d-1},$ the space of directions:
$$  h^*(v):= \max_{z_0\in \mathbb F_q^{d-1}}  \sum_{\eta\in l(z_0, v)} |h(\eta)|.$$
Let $dv$ denote the normalized surface measure on the space of directions, which assigns $q^{-(d-1)}$ to each point in $\mathbb F_q^{d-1}.$\\

For $1\le p, r\le \infty,$   we define $K(p\to r)$ to be the smallest constant such that   the estimate
\begin{equation}\label{Kdef} || h^*||_{L^r(\mathbb F_q^{d-1}, dv)} \le K(p\to r)  ||h||_{L^p(\mathbb F_q^d, dm)}\end{equation}
holds for all functions $h$ on the space $\mathbb F_q^d$ with the counting measure $dm.$
The Kakeya maximal problem over finite fields is to determine all exponents $1\le p, r\le \infty$ such that 
$$ K(p\to r)\lesssim 1.$$
This problem was initially posed by Mockenhauput and Tao \cite{MT04}, and was settled by Ellenberg, Oberlin, and Tao \cite{EOT10}, who ingenuously applied the polynomial method of Dvir \cite{Dv09}. More precisely, they proved the following critical estimate.
\begin{theorem} [Ellenberg-Oberlin-Tao, \cite{EOT10}]\label{Kdd}
Let $K(p\to r)$ be defined as in \eqref{Kdef}. Then we have
$$ K(d\to d)\lesssim 1.$$
\end{theorem}

Here, we refer the readers  to Lewko's paper \cite{Le14}, which indicates how  the sharp maximal kakeya estimate can be used to  deduce the restriction results  beyond the Stein-Tomas result for the paraboloid, where $-1$ is a square in $\mathbb F_q.$\\ 

It turned out that there is a strong connection among  the restriction estimate for $\mathcal{F}$ in $\mathbb F_q^{2d}$,  the restriction estimate for the paraboloid $P$ in $\mathbb F_q^d,$ and the Kakeya maximal estimate in $\mathbb F_q^d.$
\begin{theorem} [Theorem 9.1, \cite{MT04}]\label{PFK}
Let $P, \mathcal{F},$ and $K$ denote the paraboloid in $\mathbb F_q^d, d\ge 2,$   the flat disk in $\mathbb F_q^{2d},$ and  the Kakeya operator in $\mathbb F_q^d, respectively.$  Then,  for $ 2\le p, r \le \infty,$  we have
$$ R^*_{\mathcal{F}}(p\to r) \le R^*_{P}(2 \to r)  K\left( \left(\frac{r}{2}\right)' \to \left(\frac{p}{2}\right)' \right) ^{1/2}.$$
\end{theorem}

\section{ The Fourier transform of the surface measure on $\mathcal{F}$}
In this section, we deduce an explicit inverse Fourier transform of the normalized surface measure $d\sigma$  on the flat disk $\mathcal{F}.$  
We begin by setting  up some notation.

\begin{notation} Let $m$ denote a vector in $\mathbb F_q^{2d}$ with $d\ge 2$ an integer.  For $i=0, 1,2,3, 4, 5,$  we define
$\Omega_j \subset \mathbb F_q^{2d}$ as follows: 
\begin{itemize}
\item $\Omega_0= \{\bf{0}\}$
\item $ \Omega_1=\{m: m_d=0 =m_{2d},  m\ne \mathbf{0}  \}$
\item $\Omega_2=\{ m: m_d=0, m_{2d}\ne 0 \}$
\item $\Omega_3=\{m: m_d\ne 0,   m_{2d}=0, \delta_{\mathbf{0}} (m_{d+1}, \ldots, m_{2d-1})=0  \}$
\item $\Omega_4=\{m: m_d\ne 0,  \delta_{\mathbf{0}}(m_{d+1}, \ldots, m_{2d-1}, m_{2d})=1  \}$
\item $\Omega_5=\{ m: m_d\ne 0,  m_{2d}\ne 0  \}.$
\end{itemize}
\end{notation}

Using the above notation,  the inverse Fourier transform $(d\sigma)^{\vee}$ of the surface measure on $\mathcal{F}$ takes the following form.
\begin{proposition}\label{Promain} \label{ProF}Let $d\sigma$ be normalized surface measure on the flat disk $\mathcal{F}$ in $\mathbb F_q^n, n=2d\ge 4.$  Then, for any $m=(m_1, \ldots, m_d, m_{d+1}, \ldots, m_{2d}) \in \mathbb F_q^{n},$  we have
$$ (d\sigma)^\vee(m)=\left\{\begin{array}{ll}  1 \quad &\mbox{if}\quad m\in \Omega_0,\\
                                                                      0 \quad &\mbox{if}\quad m\in \Omega_1,\\
   q^{1-d} \chi\left( \frac{m_1m_{d+1}+  \cdots + m_{d-1} m_{2d-1}}{-m_{2d}}\right)      \quad &\mbox{if}\quad m\in \Omega_2,\\
          0 \quad &\mbox{if}\quad m\in \Omega_3,\\
               q^{\frac{1-d}{2}} \eta(m_d)^{d-1}  \eta(-1)^{(d-1)/2}\chi\left(\frac{m_1^2+ \cdots+m_{d-1}^2}{-4m_d}\right)\quad &\mbox{if}\quad m\in \Omega_4,\\
                     q^{1-d} \chi\left( \frac{m_d(m_{d+1}^2+\cdots+ m_{2d-1}^2)}{m_{2d}^2} \right)      \chi\left( \frac{m_1m_{d+1}+\cdots+ m_{d-1}m_{2d-1}}{-m_{2d}}  \right)     \quad &\mbox{if}\quad m\in \Omega_5.                    
 \end{array} \right.$$                                                                   
\end{proposition}
\begin{proof}
Let us fix $m\in \mathbb F_q^{2d}.$  Since $|\mathcal{F}|=q^{2d-2},$ we can write by the definition of $(d\sigma)^\vee$ that 
$$ (d\sigma)^\vee(m)=\frac{1}{q^{2d-2}} \sum_{x\in \mathcal{F}} \chi( m\cdot x).$$
Notice that each $x\in \mathcal{F}$ satisfies  the both equations in \eqref{defFeq}, so  whenever we fix  $$x_1, \ldots, x_{d-1},  x_{d+1}, \ldots, x_{2d-1}\in \mathbb F_q,$$    the variables $x_d$ and $x_{2d}$ are determined as in \eqref{defFeq}.  Thus,  
$(d\sigma)^\vee(m)$ is   
$$ \frac{1}{q^{2d-2}} \sum_{\substack{x_1,\ldots, x_{d-1}\in \mathbb F_q\\
                                                                                         x_{d+1}, \ldots, x_{2d-1}\in \mathbb F_q}}
                            \substack{ \chi(m_1x_1+\cdots+m_{d-1}x_{d-1} + m_d (x_1^2+\cdots+x_{d-1}^2) ) \times\\
                             \chi(m_{d+1}x_{d+1}+\cdots+ m_{2d-1}x_{2d-1}+ m_{2d} (x_1x_{d+1}+ \cdots+ x_{d-1} x_{2d-1}))}.$$
 Rearranging the general term and using Fubini's theorem, we  are also able to write  $ (d\sigma)^\vee(m)$ as 
 $$ \frac{1}{q^{2d-2}} \sum_{x_{d+1}, \ldots, x_{2d-1}\in \mathbb F_q}  \chi(m_{d+1}x_{d+1} +\cdots+ m_{2d-1}x_{2d-1} )  \prod_{i=1}^{d-1}  \sum_{x_i\in \mathbb F_q} \chi(m_dx_i^2+ (m_i+ m_{2d}x_{d+i})x_i )$$

\textbf{Case 1:} Suppose that $m_d=0$ and $m_{2d}=0$ (in this case,   $m\in \Omega_0$ or $m\in \Omega_1$).  
Then,  we see by the orthogonality of $\chi$ that 
$$(d\sigma)^\vee(m)= \delta_{\mathbf{0}}(m_1, \ldots, m_{d-1}, m_{d+1}, \ldots, m_{2d-1}).$$
Hence,   $(d\sigma)^\vee(m)=1$ for $m\in \Omega_0$, and    $(d\sigma)^\vee(m)=0$ for $m\in \Omega_1$, as required.\\

\textbf{Case 2:} Assume that $m_d=0$ and $m_{2d}\ne 0,$  which is the case when $m\in \Omega_2.$  Then it follows by the orthogonality of $\chi$  that
\begin{align*} (d\sigma)^\vee(m)&= \frac{1}{q^{2d-2}} \sum_{x_{d+1}, \ldots, x_{2d-1}\in \mathbb F_q}  \chi(m_{d+1}x_{d+1} +\cdots+ m_{2d-1}x_{2d-1} )  \prod_{i=1}^{d-1}  \sum_{x_i\in \mathbb F_q} \chi( (m_i+ m_{2d}x_{d+i})x_i )\\
&=\frac{q^{d-1}}{q^{2d-2}} \sum_{\substack{x_{d+1}, \ldots, x_{2d-1}\in \mathbb F_q:\\ x_{d+i}=-\frac{m_i}{m_{2d}}, i=1,\ldots, d-1}} \chi(m_{d+1}x_{d+1} +\cdots+ m_{2d-1}x_{2d-1} ). 
\end{align*}    
Hence,  we obtain the desired estimate
$$ (d\sigma)^\vee(m)=q^{1-d} \chi\left( \frac{m_1m_{d+1}+  \cdots + m_{d-1} m_{2d-1}}{-m_{2d}}\right).$$
                   
\textbf{Case 3:} Suppose that $m_d\ne 0$ and $m_{2d}=0,$ which corresponds to either $m\in \Omega_3$ or $m\in \Omega_4.$  It follows that
$$(d\sigma)^\vee(m)=\frac{1}{q^{2d-2}} \sum_{x_{d+1}, \ldots, x_{2d-1}\in \mathbb F_q}  \chi(m_{d+1}x_{d+1} +\cdots+ m_{2d-1}x_{2d-1} )  \prod_{i=1}^{d-1}  \sum_{x_i\in \mathbb F_q} \chi(m_dx_i^2+ m_ix_i ).$$  
Since $m_d\ne 0,$  we can invoke the formula \eqref{CSQ} to compute the product term above:
$$  \prod_{i=1}^{d-1}  \sum_{x_i\in \mathbb F_q} \chi(m_dx_i^2+ m_ix_i )= \eta^{d-1}(m_d) \mathcal{G}^{d-1} \chi\left( \frac{m_1^2+\cdots+ m_{d-1}^2}{-4m_d}\right).$$
Also note by orthogonality of $\chi$  that  the sum over $x_{d+1}, \ldots, x_{2d-1}\in \mathbb F_q$ is $q^{d-1} \delta_{\mathbf{0}}( m_{d+1}, \ldots, m_{2d-1}).$
Then it is seen that 
$$ (d\sigma)^\vee(m)=q^{-d+1}  \eta^{d-1}(m_d) \mathcal{G}^{d-1} \chi\left(\frac{m_1^2+\cdots+m_{d-1}^2}{-4m_d}\right) \delta_{\mathbf{0}}(m_{d+1}, \ldots, m_{2d-1}).$$
Thus, by the definitions of $\Omega_3$ and $\Omega_4,$ it is clear that  $(d\sigma)^\vee(m)=0$ for $m\in \Omega_3,$ and we have
$$(d\sigma)^\vee(m)=  q^{\frac{1-d}{2}} \eta^{d-1}(m_d) \eta^{(d-1)/2}(-1)\chi\left(\frac{m_1^2+ \cdots+m_{d-1}^2}{-4m_d}\right)\quad \mbox{for}~~ m\in \Omega_4, $$
where we used  the observation that  $\mathcal{G}^{d-1}= (\mathcal{G}^2)^{(d-1)/2} =  (\eta(-1)q)^{(d-1)/2}$
by Lemma  \ref{SGa}.\\

\textbf{Case 4:} Assume that  $m_d\ne 0$ and $m_{2d}\ne 0,$ which corresponds to the case when $m\in \Omega_5.$
By the formula \eqref{CSQ}, we are able to observe that 
$$ \prod_{i=1}^{d-1}  \sum_{x_i\in \mathbb F_q} \chi(m_dx_i^2+ (m_i+ m_{2d}x_{d+i})x_i )= \eta^{d-1}(m_d) \mathcal{G}^{d-1} \prod_{i=1}^{d-1}\chi\left(  \frac{(m_i+m_{2d}x_{d+i})^2}{-4m_d}\right).$$ Hence,
$ (d\sigma)^\vee(m)$ becomes
  $$\frac{1}{q^{2d-2}} \eta^{d-1}(m_d) \mathcal{G}^{d-1}\sum_{x_{d+1}, \ldots, x_{2d-1}\in \mathbb F_q}  \chi(m_{d+1}x_{d+1} +\cdots+ m_{2d-1}x_{2d-1} ) \prod_{i=1}^{d-1} \chi\left( \frac{(m_i+m_{2d}x_{d+i})^2}{-4m_d}\right).$$                      
  $$=\frac{1}{q^{2d-2}} \eta^{d-1}(m_d) \mathcal{G}^{d-1}\prod_{i=1}^{d-1}  \sum_{x_{d+i}\in  \mathbb F_q } \chi\left( \frac{(m_i+m_{2d}x_{d+i})^2}{-4m_d} +m_{d+i}x_{d+i}\right). $$      
  Applying a change of variables by  letting $y_{d+i}= m_i+ m_{2d} x_{d+i}$ for each $i=1, \ldots, d-1,$   this becomes
  $$\frac{1}{q^{2d-2}} \eta^{d-1}(m_d) \mathcal{G}^{d-1}\prod_{i=1}^{d-1}  \sum_{y_{d+i}\in  \mathbb F_q } \chi\left(  \frac{y_{d+i}^2}{-4m_d} + \frac{m_{d+i}y_{d+i}}{m_{2d}} + \frac{m_{d+i}m_i}{-m_{2d}}\right).$$
 Once again we apply the formula \eqref{CSQ},  and obtain  by a direct algebra that
 $$  (d\sigma)^\vee(m)  =  q^{-2d+2} \eta^{d-1}(-1) \mathcal{G}^{2d-2}  \chi\left( \frac{m_d(m_{d+1}^2+\cdots+ m_{2d-1}^2)}{m_{2d}^2} \right)      \chi\left( \frac{m_1m_{d+1}+\cdots+ m_{d-1}m_{2d-1}}{-m_{2d}}  \right).$$
By Lemma \ref{SGa},   notice that  $q^{-2d+2} \eta^{d-1}(-1) \mathcal{G}^{2d-2}=q^{1-d}.$  Hence, we obtain the desired estimate of $(d\sigma)^\vee(m)$ for $m\in \Omega_5.$
We have finished the proof of the lemma.                             
      
\end{proof}

\section{Proof of the $L^2\to L^r$ estimate (Theorem \ref{mainI})}
The proof will proceed by modifying  ideas from the Stein-Tomas argument.  It should be pointed out, however, that  there exist additional difficulties that arise from a lack of the Fourier decay on the flat disk.  Fortunately, the obstacles will be removed by observing certain cancellation property of the Fourier transform on some part of the domain, where the Fourier decay  is not good.\\

Notice that the H\"older conjugates of $2$ and $\frac{2n+4}{n-2}$  are $2$ and  $\frac{2n+4}{n+6},$ respectively.  Hence, the proof of Theorem \ref{mainI},  by duality in \eqref{Duality},  will be complete  once we prove the following theorem.
\begin{theorem}  \label{mainII}
Let $\mathcal{F}$ be the flat disk in $\mathbb F_q^n,  n=2d\ge 4.$ Then the restriction estimate
 \begin{equation*} ||\widehat{g}||_{L^{2}(\mathcal{F}, d\sigma)} \lesssim ||g||_{L^{\frac{2n+4}{n+6}}(\mathbb F_q^n, dm)}\end{equation*}
holds for all functions $g$ on $\mathcal{F}$.
\end{theorem}

\subsection{Standard tools from harmonic analysis}
We begin by reviewing  some skills of  harmonic analysis over finite fields, whose proofs  can be found in Been Green's lecture note \cite{G}. \\
Plancherel's theorem below can be easily proven by the orthogonality of $\chi:$ 
$$ \|\widehat{g}\|_{L^2(\mathbb F_q^d, dx)}=\|g\|_{L^2(\mathbb F_q^d, dm)}.$$
Here, one should note that  $g$ is a function on $\mathbb F_q^d$ with the counting measure $dm$, but its Fourier transform $\widehat{g}$ is defined on the dual space of $\mathbb F_q^d$ with the normalized counting measure $dx.$  Since $\mathbb F_q^d$ is isomorphic to its dual space as an abstract group,   we will use the same notation $\mathbb F_q^d$ for both  the space $\mathbb F_q^d$ and its dual space.   However, there will be no confusion since each of them has been given a distinct measure.\\

For functions $g_1, g_2$ on $(\mathbb F_q^d, dm)$,   the convolution function of $g_1$ and $g_2$  is defined on $(\mathbb F_q^d, dm):$
$$ g_1\ast g_2(m):= \sum_{m'\in \mathbb F_q^d} g_1(m-m') g_2(m').$$ 
It can be seen that $\widehat{g_1\ast g_2}=\widehat{g_1} \widehat{g_2}.$ Young's inequality for convolutions states that  if $ 1\le p_1,p_2, r \le \infty$ satisfy $1/r=1/p_1+1/p_2-1,$ then
$$ \|g_1\ast g_2\|_{L^r(\mathbb F_q^d, dm)} \le \|g_1\|_{L^{p_1}(\mathbb F_q^d, dm)} \|g_2\|_{L^{p_2}(\mathbb F_q^d, dm)} .$$

On the other hand,   if  functions $f_1, f_2$ are functions on $(\mathbb F_q^n, dx)$ with the normalized counting measure $dx$, then the convolution of $f_1$ and $f_2$ are defined on the space $(\mathbb F_q^n, dx):$
\begin{equation}\label{norconv} f_1\ast f_2(x):= \frac{1}{q^n} \sum_{y\in \mathbb F_q^n} f_1(x-y) f_2(y).\end{equation}

A powerful tool in harmonic analysis  is the interpolation theorem.
\begin{theorem} [Riesz-Thorin]\label{tmR} Let $1\le p_0, p_1, r_0, r_1\le \infty, ~  p_0\le  p_1,$ and $r_0 \le r_1.$
Suppose that $T$ is a linear operator, and satisfies  the following two estimates:
$$ \|Tg\|_{L^{r_0}} \le M_0 \|g\|_{L^{p_0}}\quad \mbox{and} \quad
\|Tg\|_{L^{r_1}} \le M_1 \|g\|_{L^{p_1}}.$$
Then we have
$$\|Tg\|_{L^{r}} \le M_0^{1-\theta} M_1^{\theta} \|g\|_{L^{p}}$$
for any $0\le \theta \le 1,$ where
$$ \frac{1}{r}=\frac{1-\theta}{r_0} + \frac{\theta}{r_1} \quad \mbox{and}\quad
\frac{1}{p}=\frac{1-\theta}{p_0} + \frac{\theta}{p_1}.$$
\end{theorem}

Now we introduce the $RR^*$ method, which is a key tool for the Stein-Tomas argument. 
Let $R, R^*$ denote the restriction operator and the extension operator  for a variety $V$ in $(\mathbb F_q^n, dx)$, namely,
$ Rg=\widehat{g}|_{V}$ and $R^*f= (fd\sigma_v)^\vee.$ The inner product of functions is defined as follows:
$$<Rg, Rg>_{L^2(V, d\sigma_v)}:= ||Rg||^2_{L^2(V, d\sigma_v)} \quad \mbox{and} \quad <g, R^*Rg>_{L^2(\mathbb F_q^n, dm)}:=\sum_{m\in \mathbb F_q^n} g(m) \overline{R^*Rg(m)}.$$
The $RR^*$ method states that  
$$<Rg, Rg>_{L^2(V, d\sigma_v)}=<g, R^*Rg>_{L^2(\mathbb F_q^n, dm)}.$$
Since $R^*Rg=  (\widehat{g} d\sigma_v)^\vee,$  the $ RR^*$ method implies that
\begin{equation}\label{RR}  ||Rg||^2_{L^2(V, d\sigma_v)}= \sum_{m\in \mathbb F_q^n} g(m) \overline{(\widehat{g} d\sigma_v)^\vee(m)}=<g, ~(\widehat{g} d\sigma_v)^\vee>_{L^2(\mathbb F_q^n, dm)}.\end{equation}

\subsection{Reduction for the proof of Theorem \ref{mainII}}
We start proving Theorem \ref{mainII}.  We aim to show that 
\begin{equation} \label{aim1}||\widehat{g}||^2_{L^{2}(\mathcal{F}, d\sigma)} \lesssim ||g||^2_{L^{\frac{2n+4}{n+6}}(\mathbb F_q^n, dm)}.\end{equation}
It can be seen, by using the inequality \eqref{RR} with $V=\mathcal{F},$  that
\begin{equation}\label{RRR} ||\widehat{g}||^2_{L^2(\mathcal{F}, d\sigma)}= <g, ~ g\ast (d\sigma)^\vee>_{L^2(\mathbb F_q^n, dm)}. \end{equation}
For $j=0, 1,2,3,4, 5,$ define
$$K_j=(d\sigma)^\vee \Omega_j.$$
Since $\sum_{j=0}^5 \Omega_j(m)= 1$  and  $K_{0}(m)=(d\sigma)^\vee(m) \delta_{\mathbf{0}}(m)=\delta_{\mathbf{0}}(m),$  we have
$$(d\sigma)^\vee= \delta_{\mathbf{0}} + \sum_{j=1}^5 K_j.$$
So the equality in \eqref{RRR} is,  by the linearity of  the inner product and the convolution,  
$$ ||\widehat{g}||^2_{L^2(\mathcal{F}, d\sigma)}= <g,~ g\ast \delta_{\mathbf{0}}>+  \sum_{j=1}^5 <g, ~ g\ast K_j>.$$
Here, and throughout,  we write a simple notation $<~~, ~~>$ for $<~~,~~>_{L^2(\mathbb F_q^n, dm)}.$
Since $g\ast \delta_{\mathbf{0}}=g,$  we observe that   
$$<g,~ g\ast \delta_{\mathbf{0}}>=||g||^2_{L^2(\mathbb F_q^n, dm)}\le ||g||^2_{L^{\frac{2n+4}{n+6}}(\mathbb F_q^n, dm)},$$
where the last inequality is valid by the nesting property of the norm with the fact that $\frac{2n+4}{n+6} \le 2.$
Also observe that $K_j\equiv 0$ for $j=1, 3,$ which is an easy consequence of Proposition \ref{ProF} and the definition of $K_j.$
It follows from these observations that, to achieve our goal in \eqref{aim1}, it will be enough to show that
$$ \sum_{j=2, 4, 5} |<g, ~ g\ast K_j>| \lesssim ||g||^2_{L^{\frac{2n+4}{n+6}}(\mathbb F_q^n, dm)}. $$
Applying  H\"older's inequality  to the general  term of the left hand side,  the matter is reduced to establishing the following:  For all $j=2, 4, 5$
\begin{equation}\label{FL2}
|| g\ast K_j||_{L^{\frac{2n+4}{n-2}} (\mathbb F_q^n, dm)}\lesssim ||g||_{L^{\frac{2n+4}{n+6}}(\mathbb F_q^n, dm)}. 
\end{equation}

In the following subsection, we will prove this inequality. Hence,  the proof of Theorem \ref{mainII} is complete, which in turn  finishes the proof of our main theorem,  Theorem \ref{mainI}.

 \subsection{Proof of the inequality \eqref{FL2}}
By interpolating  estimates of $L^2$ and $L^\infty$ norms of $g\ast K_j$,  we will deduce the inequality \eqref{FL2}.
To do this,  we begin by proving the following claim  that for each $j=2,4, 5,$  we have
\begin{equation}\label{FirC}
 ||g\ast K_j||_{L^\infty(\mathbb F_q^d, dm)} \le \max_{m\in \mathbb F_q^n} |K_j(m)| ~ ||g||_{L^1(\mathbb F_q^n, dm)},
\end{equation}
and
\begin{equation} \label{SecC}
 ||g\ast K_j||_{L^2(\mathbb F_q^d, dm)} \le  \max_{x\in \mathbb F_q^n} |\widehat{K_j}(x)|  ~||g||_{L^2(\mathbb F_q^n, dm)}.
\end{equation}

Since $||g||_{L^\infty(\mathbb F_q^n, dm)} := \max\limits_{m\in \mathbb F_q^n} |g(m)|$,  the first inequality \eqref{FirC} follows immediately from Young's inequality for convolutions. 
The second inequality \eqref{SecC} follows by  Plancherel's theorem, the property of the Fourier transform of convolution functions, and  H\"older's inequality:
$$  ||g\ast K_j||_{L^2(\mathbb F_q^n, dm)} = ||\widehat{g} \widehat{K_j}||_{L^2(\mathbb F_q^n, dx)} \le \max_{x\in \mathbb F_q^n} |\widehat{K_j}(x)|  ~||g||_{L^2(\mathbb F_q^n, dm)}.$$

\subsubsection{Proof of the inequality \eqref{FL2} for $j=2$}
First, let us estimate the value $\max_{m\in \mathbb F_q^n} |K_2(m)|.$ 
By Proposition \ref{Promain},   it is clear that $|(d\sigma)^\vee(m)|= q^{1-d}=q^{\frac{2-n}{2}}$ for all $m\in \Omega_2.$
Hence,
$$ \max_{m\in \mathbb F_q^n} |K_2(m)|= \max_{m\in \mathbb F_q^n} |(d\sigma)^\vee(m) \Omega_2(m)|=q^{\frac{2-n}{2}}.$$
So, we obtain from \eqref{FirC} with $j=2$ that
\begin{equation} \label{K2E}
||g\ast K_2||_{L^\infty(\mathbb F_q^d, dm)} \le q^{\frac{2-n}{2}} ||g||_{L^1(\mathbb F_q^n, dm)}.
\end{equation}
Next, let us estimate an upper bound of  $\max_{x\in \mathbb F_q^n} |\widehat{K_2}(x)|.$
Fix $x\in \mathbb F_q^n.$  It follows that
$$ \widehat{K_2}(x) =d\sigma \ast \widehat{\Omega_2} (x).$$
As mentioned in Introduction, we can identify $d\sigma$ as a function $q^2 1_\mathcal{F}$ on $(\mathbb F_q^n, dx).$ Hence, it can be seen  from the definition of the convolution function in \eqref{norconv} that 
\begin{equation}\label{asbefore} \widehat{K_2}(x) = q^2 \mathcal{F}\ast \widehat{\Omega_2}(x)=\frac{1}{q^{n-2}} \sum_{y\in \mathbb F_q^n} \mathcal{F}(x-y)  \widehat{\Omega_2}(y). \end{equation}
Now, by  the definition of $\Omega_2$ and the orthogonality of $\chi$,  it follows that for each $y\in \mathbb F_q^n=\mathbb F_q^{2d},$
\begin{align*} \widehat{\Omega_2}(y)=\sum_{m\in \Omega_2} \chi(-y\cdot m)
&=q^{2d-2} \delta_{\textbf{0}}(y_1,  \ldots, y_{d-1}, y_{d+1}, \ldots, y_{2d-1}) \sum_{m_{2d}\ne 0} \chi(-y_{2d} m_{2d})\\ 
& =q^{n-2} \delta_{\textbf{0}}(y_1,  \ldots, y_{d-1}, y_{d+1}, \ldots, y_{2d-1}) (q\delta_0(y_{2d}) -1).\end{align*}

Combining this estimate with the above sum for $\widehat{K_2}(x)$,    we get
\begin{align*} \widehat{K_2}(x) =&\sum_{y_d, y_{2d}\in \mathbb F_q} \mathcal{F}(x_1, \ldots, x_{d-1}, x_d-y_d, x_{d+1},\ldots, x_{2d-1},  x_{2d}-y_{2d}) (q \delta_{0}(y_{2d}) -1)\\
=& q\sum_{y_d\in \mathbb F_q} \mathcal{F}(x_1, \ldots, x_{d-1}, x_d-y_d, x_{d+1}, \ldots, x_{2d}) \\
& - \sum_{y_d, y_{2d}\in \mathbb F_q}  \mathcal{F}(x_1, \ldots, x_{d-1}, x_d-y_d, x_{d+1},\ldots, x_{2d-1},  x_{2d}-y_{2d}) .\end{align*}
Since $\mathcal{F}(\cdot)$ denotes the indicator function of the flat disk $\mathcal{F},$ we can use the trivial estimate, $|\mathcal{F}(\cdot)|\le 1.$ 
Hence,  it is obvious that 
\begin{equation}\label{K2EK}\max_{x\in \mathbb F_q^n} |\widehat{K_2}(x) | \lesssim q^2.\end{equation}
Combining this with the inequality in \eqref{SecC},  we get 
$$  ||g\ast K_2||_{L^2(\mathbb F_q^d, dm)} \lesssim  q^2||g||_{L^2(\mathbb F_q^n, dm)}. $$
Finally, interpolating this  and  the estimate in \eqref{K2E},  we obtain the required inequality \eqref{FL2} for $j=2.$\\

 \subsubsection{Proof of the inequality \eqref{FL2} for $j=5$}
 We will follow  the same steps and obtain the same estimates as in the case of $j=2.$
 To be precise, we  first claim that the following estimates  are valid:
 \begin{equation} \label{Claim1}
||g\ast K_5||_{L^\infty(\mathbb F_q^d, dm)} \lesssim q^{\frac{2-n}{2}} ||g||_{L^1(\mathbb F_q^n, dm)},
\end{equation}
and
\begin{equation} \label{Claim2}
 ||g\ast K_5||_{L^2(\mathbb F_q^d, dm)} \lesssim  q^2||g||_{L^2(\mathbb F_q^n, dm)}.
\end{equation}
 Let us assume that, for a moment,   the above claim is true. Then  interpolating the above two estimates yields the desired inequality \eqref{FL2} for $j=5.$ 
 Hence,  it remains  to prove the claimed inequalities \eqref{Claim1} and \eqref{Claim2}.
 However,  from  the inequalities \eqref{FirC} and \eqref{SecC},   it will be enough to establish the following two estimates:
 \begin{equation}\label{maxI}
 \max_{m\in \mathbb F_q^n} |K_5(m)| \lesssim q^{\frac{2-n}{2}}
 \end{equation}
 and
 \begin{equation}\label{maxII}
 \max_{x\in \mathbb F_q^n} |\widehat{K_5}(x)|\lesssim q^2.
\end{equation}


The first inequality \eqref{maxI}  can be easily proved by the same method of deriving the inequality \eqref{K2E}.
The second inequality \eqref{maxII} can be also obtained in a similar way to the argument used to deduce the inequality \eqref{K2EK}. 
Indeed, as in \eqref{asbefore},  it follows that
$$ \widehat{K_5}(x) = \frac{1}{q^{n-2}} \sum_{y\in \mathbb F_q^n} \mathcal{F}(x-y)  \widehat{\Omega_5}(y), $$
and  one can observe that  for each $y\in \mathbb F_q^n,$
$$\widehat{\Omega_5}(y)=q^{n-2} \delta_{\textbf{0}}(y_1,  \ldots, y_{d-1}, y_{d+1}, \ldots, y_{2d-1}) (q\delta_0(y_{d}) -1)(q\delta_0(y_{2d}) -1).$$
Putting  the above two equations together,   it follows by a direct computation that
$$\widehat{K_5}(x) = q^2 \mathcal{F}(x) -q\sum_{y_d\in \mathbb F_q} \mathcal{F}(x_1, \ldots, x_{d-1}, x_d-y_d, x_{d+1},\ldots, x_{2d-1},  x_{2d}) $$
$$ -q \sum_{y_{2d}\in \mathbb F_q} \mathcal{F}(x_1, \ldots, x_{d-1}, x_d, x_{d+1},\ldots, x_{2d-1},  x_{2d}-y_{2d}) $$
$$+\sum_{y_d, y_{2d}\in \mathbb F_q} \mathcal{F}(x_1, \ldots, x_{d-1}, x_d-y_d, x_{d+1},\ldots, x_{2d-1},  x_{2d}-y_{2d}). $$
Since $\mathcal{F}(\cdot)$ denotes  the indicator function of the flat disk $\mathcal{F},$
the inequality \eqref{maxII}  follows immediately  from this estimate.\\

 \subsubsection{Proof of the inequality \eqref{FL2} for $j=4$}
 We will complete the proof for $j=4$  using the similar argument as in the proof of the cases for $j=2, 5.$  However,  different kinds of estimates will be required.  Indeed,  we begin by assuming that  the following estimates hold:
 \begin{equation} \label{Claim11}
||g\ast K_4||_{L^\infty(\mathbb F_q^d, dm)} \lesssim q^{\frac{2-n}{4}} ||g||_{L^1(\mathbb F_q^n, dm)},
\end{equation}
and
\begin{equation} \label{Claim21}
 ||g\ast K_4||_{L^2(\mathbb F_q^d, dm)} \lesssim  q||g||_{L^2(\mathbb F_q^n, dm)}.
\end{equation}
Notice that  interpolating  the above two inequalities, we are able to obtain the desired  inequality \eqref{FL2} for $j=4.$
Thus, it suffices to prove the estimates \eqref{Claim11} and \eqref{Claim21}. 
However,  these estimates will follow immediately from the inequalities \eqref{FirC} and \eqref{SecC} with $j=4,$  if we are able to prove the following:
\begin{equation}\label{FirCT}
 \max_{m\in \mathbb F_q^n} |K_4(m)| \lesssim q^{\frac{2-n}{4}}
\end{equation}
and
\begin{equation} \label{SecCT}
  \max_{x\in \mathbb F_q^n} |\widehat{K_4}(x)|  \lesssim q.
\end{equation}
Hence, our problem is reduced to establishing these estimates.  Since $n=2d,$ the first inequality \eqref{FirCT}  follows by applying Proposition \ref{ProF} with the definition that $K_4=(d\sigma)^\vee \Omega_4.$
Now, we prove \eqref{SecCT}.  As in \eqref{asbefore},   we note that  for each $x\in \mathbb F_q^n,$
$$ \widehat{K_4}(x)= \frac{1}{q^{n-2}} \sum_{y\in \mathbb F_q^n} \mathcal{F}(x-y) \widehat{\Omega_4}(y).$$
Now, by the definition of $\Omega_4$ and the orthogonality of $\chi,$   we see that   for each $y\in \mathbb F_q^n,$
$$\widehat{\Omega_4}(y)=q^{d-1} \delta_{\textbf{0}}(y_1,  \ldots, y_{d-1}) (q\delta_0(y_{d}) -1).$$

Inserting this into the above equation and using the definition of $\delta_{\mathbf{0}}$,  we have
$$\widehat{K_4}(x) =\frac{1}{q^{\frac{n-2}{2}}} \sum_{y_d, \ldots, y_{2d}\in \mathbb F_q}  \mathcal{F}(x_1, \ldots, x_{d-1}, x_d-y_d, \ldots, x_{2d}-y_{2d}) (q\delta_0(y_d)-1).$$
Using a change of variables, $z_i=x_i-y_i$ for $i=d, d+1, \ldots, 2d,$   it follows that
\begin{align*}\widehat{K_4}(x) =&\frac{1}{q^{\frac{n-2}{2}}} \sum_{z_d, \ldots, z_{2d}\in \mathbb F_q}  \mathcal{F}(x_1, \ldots, x_{d-1}, z_d, \ldots, z_{2d}) (q\delta_0(x_d-z_d)-1)\\
=&\frac{q}{q^{\frac{n-2}{2}}} \sum_{z_{d+1}, \ldots, z_{2d}\in \mathbb F_q}  \mathcal{F}(x_1, \ldots, x_{d-1}, x_d, z_{d+1} \ldots, z_{2d}) \\
&-\frac{1}{q^{\frac{n-2}{2}}} \sum_{z_d, \ldots, z_{2d}\in \mathbb F_q}  \mathcal{F}(x_1, \ldots, x_{d-1}, z_d, \ldots, z_{2d}) .\end{align*}

Given $x_1, \ldots, x_{d-1}, x_d \in \mathbb F_q,$ it is not hard to observe that the following two statements hold true by the definition of the flat disk $\mathcal{F}: $ 
\begin{itemize} 
\item In the first sum above, to satisfy that  $\mathcal{F}(x_1, \ldots, x_{d-1}, x_d, z_{d+1} \ldots, z_{2d}) \ne 0,$ (namely, it is 1),   
the vector $(z_{d+1}, \ldots, z_{2d})\in \mathbb F_q^{d}$ must be contained in the plane $z_{2d}=x_1 z_{d+1} +\cdots+ x_{d-1} z_{2d-1}$ with $q^{d-1}$ elements.

\item In the second sum above,  to satisfy that $\mathcal{F}(x_1, \ldots, x_{d-1}, z_d, \ldots, z_{2d}) =1$,   
we must take $z_d=x_1^2+ \cdots+ x_{d-1}^2, $ and  the vector $(z_{d+1}, \ldots, z_{2d})\in \mathbb F_q^{d}$ must lie on the plane $z_{2d}=x_1 z_{d+1} +\cdots+ x_{d-1} z_{2d-1}$ with $q^{d-1}$ elements.
\end{itemize} 

Hence, we obtain that
$$\max_{x\in \mathbb F_q^n} |\widehat{K_4}(x)| \le \frac{q}{q^{\frac{n-2}{2}}} q^{d-1}  + \frac{1}{q^{\frac{n-2}{2}}} q^{d-1} \le 2 \frac{q}{q^{\frac{n-2}{2}}} q^{d-1} = 2 q,$$
since $n=2d.$  
Thus,  we finish the proof of the inequality  \eqref{SecCT}, as required.\\

\section{Proofs of Theorems \ref{main2} and \ref{main3}} \label{Sec5}
We begin by deducing  the relationship of the restriction phenomena between the paraboloid and the flat disk.
\begin{proposition}\label{LemFormula}
Let $P, \mathcal{F} $ denote the paraboloid in $\mathbb F_q^d, d\ge 2,$  and  the flat disk in $\mathbb F_q^{2d},$ respectively.
Suppose that $R^*_P(2\to r) \lesssim 1$ for some $\frac{2d}{d-1}\le r\le \infty.$ Then we have
\begin{equation}\label{FlatF}R^*_{\mathcal{F}}\left(\frac{2r(d-1)}{rd-r-2} \to r\right) \lesssim 1.\end{equation}
\end{proposition}
\begin{proof}
We shall invoke the following facts:
\begin{equation}\label{ResK} R^*_{\mathcal{F}}(p\to r) \le R^*_{P}(2 \to r)  K\left( \left(\frac{r}{2}\right)' \to \left(\frac{p}{2}\right)' \right) ^{1/2} \quad \mbox{for} \quad 2\le p, r\le \infty, \end{equation}
\begin{equation}\label{KaR}  K(d\to d)\lesssim 1 \quad \mbox{for all}\quad d\ge 2,\end{equation}
and
\begin{equation} \label{TrK}K(1\to \infty)\lesssim 1 \quad \mbox{for all}\quad d\ge 2.\end{equation}

The equality \eqref{ResK} is Theorem \ref{PFK} due to Mockenhaupt and Tao,  and  the estimate \eqref{KaR} is Theorem \ref{Kdd}, proven by  Ellenberg,  Oberlin, and Tao \cite{EOT10}.  The equality \eqref{TrK} is obvious as we have
$$ |f^*(v)|=\left|\sup_{m_0\in \mathbb F_q^{d-1}} \sum_{m\in \ell(m_0, v)} |f(m)|\right|\le \sum_{m\in \mathbb F_q^d} |f(m)| =||f||_{L^1(\mathbb F_q^d, dm)}.$$

Taking $p=\frac{2r(d-1)}{rd-r-2}$ in \eqref{ResK} gives us that
$$  R^*_{\mathcal{F}}\left(\frac{2r(d-1)}{rd-r-2} \to r\right) \le R^*_{P}(2 \to r)  K\left(\frac{r}{r-2}\to \frac{r(d-1)}{2} \right) ^{1/2}.$$
Since $R^*_{P}(2 \to r) \lesssim 1$ with $r\ge \frac{2d}{d-1}$ by hypothesis, to complete the proof,  it is enough to show that 
\begin{equation*} K\left( \frac{r}{r-2} \to \frac{r(d-1)}{2} \right)  \lesssim 1.\end{equation*}
But this Kakeya maximal estimate follows by interpolating the estimates \eqref{KaR} and  \eqref{TrK}. 
Thus the proof is complete. 
\end{proof}

 Notice from Proposition \ref{LemFormula}  that a better restriction result for the flat disk $\mathcal{F}$ in $\mathbb F_q^{2d}$ can be obtained by finding a $r$ index as small as possible that satisfies the $L^2\to L^r$ restriction estimate for the paraboloid $P$ in $\mathbb F_q^d.$ However, in recent years, considerably advanced results have been revealed for the $L^2\to L^r $ restriction problem for the paraboloid.
Here, we collect such results, which shall be combined with Proposition \ref{LemFormula}  to finish the proofs of Theorems \ref{main2} and \ref{main3}.\\

In even dimensions,  the following results are known.
\begin{theorem} \label{restEven}
Let $P$ be the paraboloid in $\mathbb F_q^d, d\ge 2,$ defined as in \eqref{DefP}.
Then the following estimates hold.
\begin{enumerate}
\item  $R^*_P(2\to 4)\lesssim 1 $ for $d=2.$
\item   $R^*_P\left(2\to  \frac{28}{9} \right)\lesssim 1$ for $d=4.$
\item $R^*_P(2\to 3) \lesssim 1$  if $d=4$ and $q$ is prime.
\item  $R^*_P\left(2\to  \frac{8}{3}+\varepsilon \right)\lesssim 1$ for $d=6$ and for all $\varepsilon >0.$
\item $R^*_P\left(2\to  \frac{2d+4}{d} \right)\lesssim 1$ for $d\ge 8$ even.
\end{enumerate}
\end{theorem}
The first part of the theorem  was introduced by Mockenhaupt and Tao \cite{MT04}.  More general version of the first part can be found in Theorem 1.1 in \cite{KS12}.  The second,  the fourth, and the fifth parts of the above theorem were proven by Iosevich, Koh, and Lewko  (see the proof of Theorem 1.1 in \cite{IKL20}).  The third part of the theorem due Rudnev and Shkredov was given as Theorem 1 in \cite{RS18}.

\subsection{Proof of Theorem \ref{main2}}
We are able to  complete the proof by combining  Theorem \ref{restEven} with  Proposition \ref{LemFormula}. 
Indeed,  taking  $d=2,  r=4$ in \eqref{FlatF} of Proposition \ref{LemFormula} yields the first part of Theorem \ref{main2}, namely,   $R^*_{\mathcal{F}}(4\to 4)\lesssim 1$ for $d=2.$ 
Now,  we can take $d=4, r=\frac{28}{9}$  in \eqref{FlatF} of Proposition \ref{LemFormula} so that we obtain the second part of Theorem \ref{main2}, namely,  $R^*_{\mathcal{F}}\left(\frac{28}{11}\to \frac{28}{9}\right)\lesssim 1$  for $d=4.$
On the other hand,   by  putting $d=4,  r=3$ in \eqref{FlatF} of Proposition \ref{LemFormula},  
we obtain  the third part of the theorem, which states   $R^*_{\mathcal{F}} \left(\frac{18}{7} \to 3\right)\lesssim 1$ for $d=4$ with $q$ prime.
Next,  taking $d=6, r=\frac{8}{3}+\varepsilon$ in \eqref{FlatF} of Proposition \ref{LemFormula},  we obtain the fourth part of Theorem \ref{main2}, which states $R^*_{\mathcal{F}}\left(\frac{80+30\varepsilon}{34+15\epsilon}\to \frac{8}{3}+ \varepsilon \right)\lesssim 1$ for all $\varepsilon >0.$ Finally,  when $d\ge 8$ is even,    taking $r=\frac{2d+4}{d}$  in \eqref{FlatF} of Proposition \ref{LemFormula} gives the fifth part of Theorem \ref{main2}, which is $R^*_{\mathcal{F}}\left(\frac{2d^2+2d-4}{d^2-2} \to \frac{2d+4}{d}\right)\lesssim 1$ for  $d\ge 8$ even.\\

\begin{remark}\label{rem53} The first, the third, and the fifth parts of Theorem \ref{restEven} are sharp in the sense that
each of them provides the optimal $r$ index such that $R^*_P(2\to r)\lesssim 1$ (see, for example,  Conjecture 1.2 in \cite{Ko20}).  Hence,    the first, the third, and the fifth parts of Theorem \ref{main2} are  the best possible results that can be obtained by applying Proposition \ref{LemFormula}, and new ideas are required to further improve the results.
\end{remark}

In odd dimensions, the following consequences are the best known results for the restriction estimate for the paraboloid $P$ in $\mathbb F_q^d.$
\begin{theorem} \label{restOdd}
Let $P$ be the paraboloid in $\mathbb F_q^d.$ Then the following statements are valid.
\begin{enumerate}
\item [(1)] If $d=3$ and $q\equiv 3 \pmod{4},$ then  $R^*_{P} \left( 2 \to \frac{18}{5}-\varepsilon \right) \lesssim 1$ for some $\varepsilon >0.$
\item [(2)] If $d=3$ and $q\equiv 3 \pmod{4}$ is prime,  then  $R^*_{P}\left(2 \to \frac{188}{53}+\varepsilon \right)\lesssim 1$ for all $\varepsilon >0.$
\item [(3)] If $d\ge 3 $ is odd and $q\equiv 1 \pmod{4},$ then  $R^*_{P}\left(2 \to \frac{2d+2}{d-1}\right) \lesssim 1.$
\item [(4)]  If $d=  4\ell+1$ with $\ell \in \mathbb N,$ and $q\equiv 3 \pmod{4},$ then  $R^*_{P}\left(2  \to \frac{2d+2}{d-1}\right) \lesssim 1.$
\item [(5)]   If $d=  4\ell+3$, with $\ell \in \mathbb N,$ and $q\equiv 3 \pmod{4}, $ then  $R^*_{P}\left(2 \to \frac{2d+4}{d}\right) \lesssim 1.$
\end{enumerate}
\end{theorem}
The first and the second parts of the above theorem  were given as Theorem 1 in \cite{LL13}, and Theorem 5 in \cite{Le20}, respectively. 
The third and the fourth parts are consequences of the Stein-Tomas argument (see \eqref{STT}). 
The fifth part of the theorem was given as Theorem 1.4 in \cite{KPV18}.

\subsection{Proof of Theorem \ref{main3}} As in the proof of Theorem \ref{main2},   after combining Theorem \ref{restOdd} with  Proposition \ref{LemFormula}, a direct computation  gives the desirable results.  We leave the detail to readers.

\begin{remark} It can be seen from  Conjecture 1.2 in \cite{Ko20}  that the third and the fourth parts of Theorem \ref{restOdd} are the sharp $L^2\to L^r$ restriction estimates for the paraboloid $P$ in $\mathbb F_q^d.$ Hence, their corresponding parts of Theorem \ref{main3}  are the best possible results that can be achieved by applying Proposition \ref{LemFormula}.
On the other had,  when $d\ge 3$ is odd and $-1\in \mathbb F_q$ is a square number,  it has been conjectured in \cite{Ko20} that  the results of Theorem \ref{restOdd} can be improved to the estimate
$R_P^*\left(2\to \frac{2d+2}{d-1}\right)\lesssim 1.$  In those cases,  it may be possible to get further improvement of Theorem \ref{main3}  by applying Proposition \ref{LemFormula}.
\end{remark}

\bibliographystyle{amsplain}

\end{document}